\documentclass[a4paper, 11pt]{amsart}
\usepackage{amssymb,amscd,stmaryrd,amsmath,mathabx}
\usepackage[mathcal]{eucal}
\usepackage{array,float}
\usepackage{enumitem}
\usepackage{xcolor}

\usepackage{xy}
\input xy
\xyoption{all}
\usepackage{pdflscape}
\usepackage{hyperref}
\usepackage[vcentermath]{youngtab}
\usepackage[draft]{graphicx}
\usepackage[left=3cm, right=3cm]{geometry}

\numberwithin{equation}{section}
\numberwithin{equation}{subsection}

\newtheorem{thm}{Theorem}[section]
\newtheorem{proposition}[thm]{Proposition}

\newtheorem{corollary}[thm]{Corollary}
\newtheorem{lemma}[thm]{Lemma}

\newtheorem*{remark*}{Remark}
\newtheorem{remark}[thm]{Remark}

\newcommand{\ind}{{\text{Ind}}}
\newcommand{\Hom}{{\mathrm{Hom}}}

\newcommand{\tra}{{\text{tra}}}
\newcommand{\res}{{\text{res}}}

\newcommand{\Ho}{{\mathrm{H}}}

\newcommand{\bZ}{{\mathbb Z}}
\newcommand{\bC}{{\mathbb C}}

\newcommand{\bb}[1]{\mathbb{#1}}
\newcommand{\fpp}{\mathbb{F}_p[t]/(t^2)}
\newcommand{\fpk}{{\mathbb{F}_p[t]/(t^2-k)}}
\newcommand{\hfpp}{\bb{H}_3\big(\fpp\big)}
\newcommand{\hfpk}{\bb{H}_3\big(\fpk\big)}
\newcommand{\zpp}{\mathbb{Z}/p^2 \bb{Z}}
\newcommand{\hzpp}{\bb{H}_3\big(\zpp\big)}

\newcommand{\irr}{\mathrm{Irr}}

\title[]{Projective representations of  Heisenberg \\ groups over the rings of order  $p^2$}
\author{Sumana Hatui}
\address{ SH \& EKN: Department of Mathematics,
	Indian Institute of Science,
	Bangalore 560012, India. }

\email{sumanahatui@iisc.ac.in}

\author{E. K. Narayanan}

\email{naru@iisc.ac.in}

\author{Pooja Singla}
\address{ PS: Department of Mathematics and Statistics, Indian Institute of Technology Kanpur, Kanpur 208016, India. }
\email{psingla@iitk.ac.in}
\begin{document}

\begin{abstract}
In this article we describe the 2-cocycles, Schur multiplier and representation group of discrete Heisenberg groups over the unital rings of order $p^2$. We describe all projective representations of Heisenberg groups with entries from the rings $\mathbb Z/p^2\mathbb Z$ and $ \fpp$ and obtain a classification of their degenerate and  non-degenerate 2-cocycles.
\end{abstract}


\maketitle

\section{Introduction}
The theory of projective representations of finite groups  was first studied by Schur in a series of papers \cite{IS4,IS7,IS11}.
A projective representation  of a group $G$ is a homomorphism from $G$ to projective general linear group $\mathrm{PGL}(V)$, where $V$ is a complex vector space. So $\rho:G \to \mathrm{GL}(V)$ is a map such that $\rho(1)=id_V$ and  there is a $2$-cocycle $\alpha:G \times G \to \mathbb C^\times$ satisfying  $$\rho(xy)=\alpha(x,y)\rho(x)\rho(y).$$ Then, we say $\rho$ is an $\alpha$-representation.

 For cyclic groups the irreducible projective representations are same as the ordinary representations (up to equivalence) and hence one dimensional. However, this is not true in general for abelian groups. For abelian groups, these have been studied by several authors in  \cite{RJ}, \cite{Moa}, \cite{Moa2}.

Recently, the first named and third named authors studied the projective representations of the discrete Heisenberg group over cyclic rings in \cite{PS}. This paper may be considered as a continuation of \cite{PS}, where we study the projective representations of discrete Heisenberg groups over the unital rings $R$ of order $p^2,$ where $p$ is an odd prime. The discrete Heisenberg group of rank one over $R$, denoted by $\mathbb{H}_{2n+1}(R)$, is the set $ R \times R^n \times R^n$ with the multiplication given by
\[
\begin{array}{l}
(c_1, b_1, \ldots, b_n, a_1, \ldots, a_n) (c'_1, b'_1, b'_2, \ldots, b'_n, a'_1, a'_2, \ldots, a'_n) \\
= (c_1+c'_1+ \sum_{i=1}^n a_i*b_i', a_1 + b'_1, \ldots, b_n + b'_n, b_1 + a'_1, \ldots, a_n + a'_n),
\end{array}
\]
where $a*b$ denotes the multiplication in $R$.

Now onwards, we consider $R$ to be a ring of order $p^2$ with unity. It is easy to prove that any such ring $R$ is commutative and  is isomorphic to one of the following:

\begin{itemize}

\item $\mathbb Z/p^2\mathbb Z,$

\item $\mathbb Z/p\mathbb Z \times \mathbb Z/p\mathbb Z,$

\item $\fpp,$

\item $\mathbb{F}_{p^2}.$
\end{itemize}
 We remark that the results in \cite{PS} take care of the case $\mathbb H_{2n+1}(R) $ for $n > 1.$ Indeed, one of the main results in \cite{PS} is that, for $n>1,$ every irreducible projective representation of $\mathbb{H}_{2n+1}(R)$ (where $R$ is any commutative ring) is obtained from an irreducible projective representation of the abelian group $R^{2n}$ via inflation. Therefore, in this article we will focus on the case $n = 1,$ that is on $\mathbb H_3(R)$.

In the study of the projective representations of a finite group $G$, the main ingredients are to describe the Schur multiplier of $G$, to determine a representation group $G^\star$ of $G$ and then describe the ordinary representations of $G^\star$. We refer the reader to \cite{Karpilovsky, GK85} for any unexplained terms or notation in this article.
Now, let $R$ be one of the rings mentioned above, of order $p^2.$ Main results obtained in this article can be summarized as follows:

\begin{itemize}

\item[(a)] Description of Schur multiplier $\Ho^2\big(\mathbb{H}_{3}(R),\bC^\times\big),$
\item[(b)] Explicit description of $2$-cocycles of $\mathbb H_3(R),$
\item[(c)] Construction of a representation group $G^{\star}$ for $\mathbb H_3(R)$ and its irreducible representations,
\item[(d)] Description of non-degenerate $2$-cocycles of $\mathbb{H}_{3}\big(\fpp \big)$ and $\mathbb{H}_{3}(\bZ/p^2\bZ).$
\end{itemize}

We end this section by providing the statements of the results:

\begin{thm} \label{SchurM}
Suppose $p$ is an odd prime and $R$ be a unital ring of order $p^2$, where $p$ is an odd prime. Then
		\[
		\Ho^2\big(\mathbb{H}_{3}(R),\bC^\times\big)=
		\begin{cases}
		(\bZ/p\bZ)^{8} & \text{if } R \cong \fpp,\\
		(\bZ/p\bZ)^{8} & \text{if } R \cong \mathbb{F}_{p^2},\\
	        (\bZ/p\bZ)^{8} & \text{if }  R \cong \mathbb Z/p\mathbb Z \times \mathbb Z/p\mathbb Z,\\
		(\bZ/p^2\bZ)^{2} & \text{if } R\cong \mathbb Z/p^2\mathbb Z.\\
		\end{cases}
		\]
	\end{thm}	
The proof of this result is given in Section \ref{Schur-Multiplier}. For $\mathbb{H}_3(\bZ/p^2\bZ)$, this result follows from  \cite[Theorem 1.1]{Urban}. For the remaining cases, we mainly use the results of Blackburn and Evens \cite{Blackburn}. Since $G = \mathbb H_3(R)$ is a $p$-group of nilpotency class $2$ with the property that $G/G'$ (where $G'$ is the commutator subgroup of $G$) is an elementary abelian group, we can apply the main results from \cite{Blackburn}.
Our next result is the description of the 2-cocycles of the group $\mathbb H_3(R).$  We prove the following for $\mathbb{H}_{3}\big(\fpp\big)$.

\begin{thm} \label{2cocycle}
	Up to cohomologous, every $2$-cocycle $\alpha$ of $\mathbb{H}_{3}\big(\fpp\big)$ is of the following form:
	\begin{eqnarray*}
		&&\alpha\big(x_1^{a_1}x_2^{a_2}y_1^{b_1}y_2^{b_2}z_1^{c_1}z_2^{c_2}, x_1^{a_1'}x_2^{a_2'}y_1^{b_1'}y_2^{b_2'}z_1^{c_1'}z_2^{c_2'}\big) =\mu_1^{a_2a_1'} \mu_2^{b_2a_1'}\mu_3^{b_2a_2'} \mu_4^{b_2 b_1'}  \mu_5^{b_1{a_1' \choose 2}+a_1'c_1}\\
		&&\hspace{5cm}  \mu_6^{b_2{a_1' \choose 2}+a_1'c_2+a_2'c_1} \mu_7^{ a_1' {b_1\choose 2}-b_1c_1'} \mu_8^{a_2'{b_1\choose 2}-b_2c_1'-b_1c_2'},
	\end{eqnarray*}
	where  $\mu_i \in \bC^\times$ such that $\mu_i^p=1$.	
\end{thm}
Here $x_1^{a_1}x_2^{a_2}y_1^{b_1}y_2^{b_2}z_1^{c_1}z_2^{c_2}$ denotes a general element in the group $\mathbb{H}_{3}\big(\fpp\big).$ See subsection \ref{subsection-presentation} for more details.
The description of $2$-cocycles of $\mathbb{H}_{3}(\bZ/p^2\bZ )$ has already appeared in \cite{PS}.
For other groups, the description of $2$-cocycles is given in Section \ref{section-cocycle-description}.


  As is well known, the projective representations of a group are obtained from ordinary representations of its representation group. To study the projective representations of $\mathbb H_3(R)$ we construct a representation group of $\mathbb H_3(R).$ In this direction, we have the following result for $\mathbb{H}_{3}\big(\fpp\big)$. For a group $G$ and $x,y \in G$, the commutator $x^{-1}y^{-1}xy$ is denoted by $[x,y]$. Whenever we write a presentation of a group, we assume all the commutators $[x,y]$ for generators $x,y$,  which are not explicitly stated in the presentation, are trivial.

\begin{thm}\label{representation-group}
A representation group of $G = \mathbb{H}_{3}\big(\fpp\big)$ is given by
\begin{eqnarray*}
	G^\star&=&\langle  x_i,y_i , i=1,2 \mid [y_1,x_1]=z_1,  [y_1,x_2]=z_2, [x_2,x_1]=w_1,[y_2,x_1]z_2^{-1}=w_2,
	\\
	&&  [y_2,x_2]=w_3,[y_2,y_1]=w_4,  [z_1, x_i]=v_i,  [z_1, y_i]=u_i, [z_2,x_1]=[z_1, x_2],   \\
	&&  [z_2,y_1]=[z_1, y_2], x_i^p=y_i^p=z_i^p=w_j^p=1, 1 \leq j \leq 4  \rangle.
\end{eqnarray*}
\end{thm}
The parallel results for $\mathbb{H}_3(\bZ/p^2 \bZ)$ have already appeared in \cite{PS}.
For other groups, a description of their representation group is included in Section \ref{section-representation group}. In Section \ref{Non-degenerate-cocycles} we give a construction of all projective representation of  $\mathbb{H}_{3}\big(\fpp\big)$ and $\mathbb{H}_{3}\big(\mathbb Z/p^2\mathbb Z\big)$. Following \cite{MR2534251}, we say a 2-cocycle $[\alpha] \in \Ho^2(G, \mathbb C^\times)$ is  non-degenerate if the twisted group algebra $\mathbb C^\alpha[G]$ is a simple algebra. In this case, the group $G$ is called of central type. These groups  play an important role in the classification of semisimple triangular complex Hopf algebras, see \cite{Etingof-gelaki2000}. We use
our construction of projective representations to classify the degenerate and non-degenerate cocycles of $\mathbb{H}_{3}\big(\fpp\big)$ and $\mathbb{H}_{3}\big(\mathbb Z/p^2\mathbb Z\big)$. More specifically we prove the following result.
\begin{thm}
	\label{thm:number-non-degenerate} For a finite group $G$, let $X(G) = \{ [\alpha] \in \Ho^2(G, \mathbb C^\times) \mid [\alpha] \mathrm{\,\,is \,\, non \,\, degenerate} \}$. Then
	\[
	|X(G)| = \begin{cases}
	 p^2(p-1)^2, & G = \mathbb{H}_{3}\big(\mathbb Z/p^2\mathbb Z\big)  \\
	    p^5(p-1)^2(p+1), &  G = \mathbb{H}_{3}\big(\fpp\big).	
	\end{cases}
	\]
\end{thm}

For a proof of the above, see Section \ref{Non-degenerate-cocycles}. In fact, we obtain the above from a construction of the irreducible representations of $G^{\star},$ see Sections~ \ref{nondegco1} and ~\ref{nondegco2}.

\section{Preliminaries}
In this section we fix the notation and recall a few results which will be used in the upcoming sections. For a group $G$, we use $G'$ to denote its commutator subgroup and $Z(G)$ denotes the center of $G$. The first result, due to Schur, describes the Schur multiplier of a direct product in terms of the Schur multipliers of its components, see \cite[Theorem 2.2.10]{Karpilovsky}.
\begin{thm}\label{directproduct}
For finite groups $G_1$ and $G_2$, $$\Ho^2(G_1\times G_2,\mathbb C^\times)\cong \Ho^2(G_1,\mathbb C^\times)\times \Ho^2( G_2,\mathbb C^\times) \times \Hom(\frac{G_1}{G_1'}\otimes \frac{G_2}{G_2'},\mathbb C^\times).$$
\end{thm}

A central extension,
\begin{equation}\label{equation-stem}
 1\to A \to G \to G/A \to 1
\end{equation} is called a {\it stem} extension, if $A \subseteq Z(G) \cap G'$. For a central extension (\ref{equation-stem}), the map $\tra:\Hom( A,\bC^\times) \to \Ho^2(G/A, \bC^\times) $ given by $f \mapsto [\tra(f)]$,
$$\tra(f)(\overline{x},\overline{y}) = f(\mu (\overline{x})\mu(\bar{y})\mu(\bar{xy})^{-1}),\,\, \overline{x}, \overline{y} \in G/A, $$
for a section $\mu: G/A \rightarrow G$ is a group homorphism and is called the transgression homomorphism. The map $\inf : \Ho^2(G/A, \bC^\times)   \to  \Ho^2(G, \bC^\times) $ given by $[\alpha] \mapsto [\inf(\alpha)]$, where $\inf(\alpha)(x,y) = \alpha(xA,yA)$, is a group homomorphism and is called the inflation homomorphism, The spectral sequence
for cohomology of groups yields the following exact sequence.
\begin{equation}\label{exact-sequence}
1 \rightarrow \Hom( A,\bC^\times) \xrightarrow[]{\tra} \Ho^2(G/A, \bC^\times)  \xrightarrow{\inf} \Ho^2(G, \bC^\times)  \xrightarrow{(\res,\chi)}   \Ho^2(A, \bC^\times) \oplus \Hom(G/G' \otimes A,  \bC^\times),
\end{equation}
where the map $\chi$, defined by  Iwahori and Matsumoto  \cite{IM}, is given by $\chi([\alpha])(gG'\otimes a)=\alpha(g,a)\alpha(a,g)^{-1}$ for $g\in G, a \in A$.

\begin{lemma}[Hall–Witt identity]
Let $G$ be a finite group of nilpotency class $3$. For $x,y,z \in G$,  we have
	\[
	[x,y^{-1},z][y,z^{-1},x][z,x^{-1},y]=1.
	\]
\end{lemma}

\subsection{Projective representations of a finite group} In this section we include the results that we require regarding the projective representations of a finite group. We will use these in Section~\ref{Non-degenerate-cocycles}.

Let $G$ be a finite group. We use $Z^2(G, \mathbb C^\times)$ to denote the set of all $2$-cocycles of $G$. For $\alpha \in Z^2(G, \mathbb C^\times)$,  we use $\mathrm{Irr}^\alpha(G)$ to denote the set of equivalence classes of irreducible $\alpha$-representations of $G$. For $\alpha = 1$, we use $\mathrm{Irr}(G)$ instead of $\mathrm{Irr}^\alpha(G)$ and call these to be the set of ordinary irreducible representations of $G$. Let $G^\star$ be a representation group of $G$ with $A\cong \Ho^2(G, \mathbb C^\times)$ such that
\begin{eqnarray}
	\label{stem-extension-G}
	1 \rightarrow A \rightarrow G^\star \rightarrow  G \rightarrow 1
\end{eqnarray}
is a stem extension. The following well known result relates the projective representations of $G$ and the ordinary ones of $G^\star$.
\begin{thm}\label{inf}
	Let $\alpha$ be a $2$-cocycle of $G$. Let $\chi \in \Hom(A,\mathbb C^\times)$ be such that $\mathrm{tra}(\chi) = [\alpha]$.
	There is a bijective correspondence between
	$$\irr^{\alpha}(G)\leftrightarrow \irr{(G^\star \mid \chi )}$$ obtained  via inflation. In particular, we obtain the following bijection via inflation.
	$$\bigcup_{[\alpha] \in \Ho^2(G,\bC^\times)} \irr^{\alpha}(G)\leftrightarrow \irr{(G^\star )}$$
	
\end{thm}
A proof of this follows from the proof of \cite[Chapter 3, Section 3]{GK85}, see also \cite[Theorem~3.2]{PS}. Therefore, to determine the projective representations of   $\hfpp$ and $\hzpp$, it suffices to determine the ordinary representations of their representation groups. We first discuss a general method that will work in our situation.

Let $G$ be a finite group with an abelian normal subgroup $N$ such that $G/ N $ is abelian. Let $\chi: N \rightarrow \mathbb C^\times$ be a one dimensional irreducible representation of $N$ and let  $I_G(\chi) = \{ g \in G \mid \chi^g = \chi\}$ be the inertia group of $\chi$ in $G$. Let $\mathrm{Irr}(G \mid \chi)$ denote the set of inequivalent irreducible representations of $G$ lying above $\chi$, that is $\rho \in
\mathrm{Irr}(G \mid \chi)$ if and only if $\mathrm{Hom}_{N}(\rho|_{N}, \chi)$ is non-trivial.

By Clifford theory, the map $\rho \mapsto \ind_{I_G(\chi)}^{G}(\rho)$ gives a bijection between $\mathrm{Irr}(I_G(\chi) \mid \chi)$ and $\mathrm{Irr}(G \mid \chi)$. Therefore it suffices to determine $\mathrm{Irr}(I_G(\chi) \mid \chi)$ for every $\chi$. Below we mention a method that helps us in this direction.

Let $K$ be a finite group with an abelian normal subgroup $N$ such that $K/ N $ is abelian. Let $\chi: N \rightarrow \mathbb C^\times$ be a one dimensional irreducible representation of $N$ such that $\chi^k = \chi$ for all $k \in K$, that is $I_K(\chi) = K$. Let $T$ be a fixed set of left coset representatives of $N$ in $K$. Define a map $\chi': K \rightarrow \mathbb C^\times$ by $\chi'(kn) = \chi(n)$ for all $k\in T$ and $n \in N$. Following \cite[Chapter~11]{Isaacs-book}, let $\alpha \in Z^2(K, \mathbb C^\times)$ be a 2-cocycle of $K$ associated to $\chi'$ and $\beta \in Z^2(K/N, \mathbb C^\times)$ be defined by $\beta(gN, hN) = \alpha(g,h)$ for $g,h \in K$.
\begin{lemma}
	\label{lemma:extension-condition}
	The following are equivalent.
	\begin{enumerate}
		\item The character $\chi$ extends to $K$.
		\item $[K,K] \subseteq \mathrm{Ker}(\chi)$.
		\item $[\beta] = 1$.
	\end{enumerate}
\end{lemma}
\begin{proof}
	The equivalence of (1) and (2) follows from the fact that $[K, K] \subseteq N$ and $N/[K,K]$ is an abelian group. The equivalence of (1) and (3) follows from \cite[Theorem~11.7]{Isaacs-book}.
\end{proof}
With the above notations, the following result can be obtained from \cite[Theorem~4.2]{GK85}.
\begin{lemma}
	\label{lemma:quotient-group-bijection}
	There exists a dimension preserving bijection between the sets $\mathrm{Irr}(K \mid \chi)$ and $\irr^{\beta^{-1}}(K/N)$.
\end{lemma}
Hence it boils down to understand the projective representations of the quotient group $K/N$. For our case, this quotient group will turn out to be an abelian group. The projective representations of abelian groups are well studied, and we use some of these results. In particular, we will use the following lemma without further reference.
\begin{lemma}
	\label{lemma:dimension-of-abelian-representations}
	Let $G$ be a finite abelian group. Then the following hold:
	\begin{enumerate}
		\item For any fixed $\alpha \in  Z^2(G, \mathbb C^\times)$, all irreducible representations in $\irr^\alpha(G)$ have equal dimension.
		\item For $|G| \in \{p^2, p^3\}$ and $1 \neq [\alpha ] \in \Ho^2(G, \mathbb C^\times)$, every $\rho \in \irr^\alpha(G)$ has dimension $p$.
		\item For $|G| = p^4$ and $1 \neq [\alpha ] \in \Ho^2(G, \mathbb C^\times)$, every $\rho \in \irr^\alpha(G)$ has dimension either $p$ or $p^2$.
	\end{enumerate}
	
\end{lemma}
\begin{proof} Here (1) follows from \cite[Theorem~1]{Backhouse1972}. The assertions (2) and (3) follow from the fact that any $\rho \in \irr^\beta(G)$ for $[\beta] \neq 1$ satisfies $1 < \dim(\rho) \leq \sqrt{|G|}$ and $\dim(\rho)$ divides the order of $G$.	
\end{proof}

At this point, we mention the ways to identify the non-degenerate cocycles of a group $G$ by using the ordinary representations of its representation group $G^\star$.  Consider the stem extension (\ref{stem-extension-G}). Let  $\alpha$ be a $2$-cocycle of $G$ and $\chi \in \Hom(A,\mathbb C^\times)$  such that $\mathrm{tra}(\chi) = [\alpha]$.
\begin{lemma}
	\label{lemma:non-degenerate-cocycles}
	The following conditions are equivalent.
	\begin{enumerate}
		\item The 2-cocycle $\alpha$ is non-degenerate.
		\item $|\mathrm{Irr}^\alpha(G)| = 1$
		\item There exists $\rho \in \mathrm{Irr}^\alpha(G)$ such that $\dim(\rho) = \sqrt{|G|}$.
		\item  There exists $\rho \in \mathrm{Irr}(G^\star \mid \chi)$ such that $\dim(\rho) = \sqrt{|G|}$.
		\item $|\mathrm{Irr}(G^\star \mid \chi)| = 1$.
	\end{enumerate}
\end{lemma}
\begin{proof}
	Equivalence of (1), (2), and (3) follows from the definition of a non-degenerate $2$-cocycle. Equivalence of (2) and (5) as well as of (3) and (4) follows from Theorem~\ref{inf}.
\end{proof}
We will use this result to classify the non-degenerate cocycles of $\hfpp$ and $\hzpp$.

\subsection{Presentation and matrix form of groups $\mathbb H_3(R)$}\label{subsection-presentation}

The groups $\hfpp$ is of nilpotency class $2$ and have the following presentation:
\[
\hfpp=\langle  x_1, x_2, y_1, y_2 \mid [y_1,x_1]=z_1,  [y_1,x_2]=[y_2,x_1]=z_2, x_i^p=y_i^p=1, i=1,2   \rangle.
\]

For $\mathbb{F}_{p^2}$, let $k \in F_{p^2} \setminus \mathbb F_{p}$. Then $\mathbb{F}_{p^2} \cong \mathbb{F}_p[t]/(t^2-k)$ and therefore
\[
\bb{H}_{3}(\mathbb{F}_{p^2}) \cong \bb{H}_{3}(\mathbb{F}_p[t]/(t^2-k)).
\]
Now onwards, we will keep the $k$ as above fixed for $\mathbb{F}_{p^2}$ and use this wherever required. We have the following presentation of $\hfpk$:

\begin{eqnarray*}
	{\bb{H}}_{3}\Big(\frac{\mathbb{F}_p[t]}{(t^2-k)}\Big)=\langle  x_1, x_2, y_1, y_2\mid [y_1,x_1]=z_1, [y_2,x_2]=z_1^k,   [y_1,x_2]=[y_2,x_1]=z_2, \\
	x_i^p=y_i^p=1, i=1,2   \rangle.
\end{eqnarray*}

It is helpful to think of $\bb{H}_3(R)$ in its matrix form, that is as a group of $3 \times 3$ matrices with entries from the ring $R$. In the above two presentations, the element $x_1^{a_1}x_2^{a_2}y_1^{b_1}y_2^{b_2}z_1^{c_1}z_2^{c_2}$ corresponds to
$$
\begin{pmatrix}
	1 & (b_1,b_2) & (c_1,c_2)\\
	0 & 1& (a_1,a_2)\\
	0 & 0& 1
\end{pmatrix}
$$
in the matrix form of $\bb{H}_3(R)$.

\subsection{Schur multiplier of certain $p$-groups}
\label{section: Schur-multiplier-p-groups}
Let  $G$ be a $p$-group of nilpotency class $2$ with elementary abelian $G/G'$. In this section we recall the theory given in \cite[Section 3]{Blackburn} to compute the Schur multiplier of $G$.

Consider $G/G'$ and $G'$ as vector spaces over $\mathbb F_p$, denoted by $V$ and $W$ respectively.
For $g_1,g_2,g_3\in G$, let $U_1$ be the subspace of $V\otimes W$ spanned by the elements of the form
$$
\bar{g}_1\otimes [g_2,g_3]+ \bar{g}_2\otimes [g_3,g_1]+\bar{g}_3\otimes [g_1,g_2],
$$
where $\bar{g}_i=g_i G' \in V$, for $ i=1,2,3$. Let $U_2$ be the subspace of  $V\otimes W$ spanned by all $\bar{g}_1\otimes g_1^p$.
Now consider $U =U_1+U_2$.

We have the following result from  \cite[Theorem 3.1]{Blackburn}.
\begin{proposition}\label{blackburn}
Let $G$ be a $p$-group of nilpotency class $2$ such that $G/G'$ is  elementary abelian. Then
 $|\Ho^2(G,\mathbb C^\times)|=|\frac{V \wedge V}{W}||\frac{V \otimes W}{U}|$.
\end{proposition}
We will continue to use these notations in the next section.

\begin{lemma}
	\label{lemma: elementary-abelian-shur-multiplier}
Let $p$ be an odd prime and $G$ be a $p$-group of exponent $p$, of nilpotency class $2$. Then $\Ho^2(G,\mathbb C^\times)$ is an elementary abelian $p$-group.
\end{lemma}
\begin{proof}
	If $G$ has a free presentation $F/\mathcal{R}$, then $\Ho^2(G,\mathbb C^\times) \cong \Ho_2(G,\mathbb Z)\cong  \frac{F'\cap \mathcal{R}}{[F,\mathcal{R}]}$, follows from \cite[Theorem 2.4.6]{Karpilovsky}.
For $x,y \in F$ and for odd $p$,
$$1 \equiv [x^p,y] \equiv [x,y]^p[[x,y],x]^{p\choose 2} \equiv [x,y]^p~\pmod{[F,\mathcal{R}]}$$
imply that  $(F')^p \subset [F,\mathcal{R}]$.  Therefore, for odd $p$,  $\Ho^2(G,\mathbb C^\times)$ is an elementary abelian $p$-group.
\end{proof}

The following result will help us to describe the 2-cocycles of the group $\hfpp$.
\begin{proposition}\label{exponent-p}
	Let $p$ be an odd prime and $G$ be a $p$-group of exponent $p$ and of nilpotency class $2$ such that there is a central subgroup $Z\subseteq Z(G)\cap G'$ with the property
	$$\big |\Ho^2(G,\mathbb C^\times)\big |=\Big |\frac{\Ho^2(G/Z,\mathbb C^\times)}{\Hom(Z,\mathbb C^\times)}\Big|\big|\Hom\big(G/G'\otimes Z,\mathbb C^\times\big)\big|.$$
	Then
	$$\Ho^2(G,\mathbb C^\times) \cong \frac{\Ho^2(G/Z,\mathbb C^\times)}{\Hom(Z,\mathbb C^\times)} \times \Hom\big(G/G'\otimes Z,\mathbb C^\times\big).$$
\end{proposition}
\begin{proof}
By \eqref{exact-sequence}, we have the following exact sequence
	$$
	1 \to \frac{\Ho^2(G/Z,\mathbb C^\times)}{\Hom(Z,\mathbb C^\times)}  \xrightarrow{\inf}   \Ho^2(G,\mathbb C^\times) \xrightarrow{\chi} \Hom\big(G/G'\otimes Z,\mathbb C^\times\big) \to 1.
	$$
	The group $\Ho^2(G,\mathbb C^\times)$ is an elementary abelian $p$-group by Lemma~\ref{lemma: elementary-abelian-shur-multiplier}. Therefore the result follows.
\end{proof}


\section{Schur multiplier of $\mathbb H_3(R)$}\label{Schur-Multiplier}

In this section we prove Theorem \ref{SchurM}.

\begin{proof}[Proof of Theorem \ref{SchurM}.]
By Lemma \ref{lemma: elementary-abelian-shur-multiplier} it follows that, for odd prime $p$,  $\Ho^2(G,\mathbb C^\times)$ is an elementary abelian $p$-group.

\subsection{\bf $G=\hfpp$} This is a nilpotent group of nilpotency class $2$ and is of exponent $p$. By  using the notations of Section~\ref{section: Schur-multiplier-p-groups}, we have $U_2=1$ and so $U=U_1$.
Now from Proposition \ref{blackburn},
$$|\Ho^2(G,\mathbb C^\times)|= |\frac{(G/G' \wedge G/G')}{G'}||\frac{(G/G' \otimes G')}{U}|=\frac{p^{12}}{|U|},$$
where $U$ is the subspace of $(G/G' \otimes G')$ generated by the elements of the form
$$\bar{x}_1  \otimes [x_2,x_3]+ \bar{x}_2 \otimes[x_3,x_1]+\bar{x}_3\otimes [x_1,x_2] \text{ for }  \bar{x}_i=x_iG'\in G/G'.$$
Using the presentation of $G$, we have
\begin{eqnarray*}
	U  &=& \langle \bar{y}_1\otimes [y_2,x_1]+\bar{y}_2\otimes [x_1,y_1],   \bar{x}_1\otimes [y_1,x_2]+\bar{x}_2 \otimes [x_1,y_1], \\
	&&\bar{y}_2 \otimes [y_1,x_2] , \bar{x}_2\otimes [y_2,x_1]\rangle\\
&=&\langle  (\bar{y}_1\otimes z_2+ \bar{y}_2\otimes z_1^{-1}) ,
(\bar{x}_1\otimes z_2+\bar{x}_2\otimes z_1^{-1}), \bar{y}_2\otimes z_2, \bar{x}_2\otimes z_2  \rangle.
\end{eqnarray*}
Therefore  $|U|=p^4$. By Lemma~\ref{lemma: elementary-abelian-shur-multiplier},
$$
\Ho^2(G, \mathbb C^\times) \cong (\mathbb Z/p\mathbb Z)^8.$$

\begin{remark} The similar method works for $\Ho^2\Big(\mathbb{H}_{3}\big(\frac{\mathbb{F}_p[t]}{( t^r)}\big),\mathbb C^\times\Big)$ for $r \in \mathbb N$. One can prove that for  $r \in \mathbb N$,
$$\Ho^2\Big(\mathbb{H}_{3}\big(\frac{\mathbb{F}_p[t]}{( t^r)}\big),\mathbb C^\times\Big) \cong (\mathbb Z/p\mathbb Z)^{2r^2}. $$
\end{remark}

\subsection{\bf $G=\mathbb{H}_{3}(\mathbb{F}_{p^2})$}

By Lemma~\ref{lemma: elementary-abelian-shur-multiplier} and Proposition \ref{blackburn}, $\Ho^2(G,\mathbb C^\times)$ is an elementary abelian group and
$$|\Ho^2(G,\mathbb C^\times)|= \frac{p^{12}}{|U|},$$
where $$U=\big\langle \big(\bar{y}_1\otimes z_1^k+\bar{y}_2 \otimes z_2^{-1}\big), \big(\bar{y}_2\otimes z_1+ \bar{y}_1 \otimes  z_2^{-1}\big), \big(\bar{x}_1\otimes z_1^k+\bar{x}_2\otimes z_2^{-1}\big),  \big(\bar{x}_2\otimes z_1+ \bar{x}_1 \otimes  z_2^{-1} \big)\big\rangle.$$
Therefore $|U|=p^4$
and  we have $$\Ho^2\Big(\mathbb{H}_{3}\big(\frac{\mathbb{F}_p[t]}{(t^2-k)}\big),\mathbb C^\times\Big) \cong (\mathbb Z/p\mathbb Z)^8.$$

\subsection{$\mathbb{H}_{3}(\mathbb Z/p\mathbb Z \times \mathbb Z/p\mathbb Z)$ {and}  $\mathbb{H}_{3}(\mathbb Z/p^2\mathbb Z)$ }
Since $\mathbb{H}_{3}(\mathbb Z/p\mathbb Z \times \mathbb Z/p\mathbb Z) \cong \mathbb{H}_{3}(\mathbb{Z}/p\mathbb{Z})\times \mathbb{H}_{3}(\mathbb{Z}/p\mathbb{Z})$, result follows from Theorem \ref{directproduct}. For the group  $\mathbb{H}_{3}(\mathbb Z/p^2\mathbb Z)$, result follows from \cite[Theorem 1.1]{Urban}.
\end{proof}


\section{Representation group of $\mathbb H_3(R)$}\label{section-representation group}
\subsection{$G=\mathbb{H}_{3}\big( \fpp \big)$}
In this section, we prove Theorem \ref{representation-group}.

	\begin{proof}[Proof of Theorem \ref{representation-group}.]
Recall that $G^\star$ is given by the following:
		\begin{eqnarray}\label{G*}
			G^\star&=&\langle  x_i,y_i , i=1,2 \mid [y_1,x_1]=z_1,  [y_1,x_2]=z_2, [x_2,x_1]=w_1,[y_2,x_1]z_2^{-1}=w_2,
			\nonumber \\
			&&  [y_2,x_2]=w_3,[y_2,y_1]=w_4,  [z_1, x_i]=v_i,  [z_1, y_i]=u_i, [z_2,x_1]=[z_1, x_2],  \nonumber \\
			&&  [z_2,y_1]=[z_1, y_2], x_i^p=y_i^p=z_i^p=w_j^p=1, 1 \leq j \leq 4  \rangle.
		\end{eqnarray}
		If $G^\star$ is  a group of order $p^{14}$, then we have the following stem extension
		$$1 \to Z \to G^\star\to G\to 1$$
		for $Z=\langle u_i,v_i,w_j, 1\leq i \leq 2, 1 \leq j \leq 4 \rangle\cong (\mathbb Z/p\mathbb Z)^8$.
		Since $\Ho^2(G,\mathbb C^\times)\cong (\mathbb Z/p\mathbb Z)^8$,  $G^\star$ is a representation  group of $G$. We now proceed to prove that $G^\star$ is of order $p^{14}$.

	Consider the free group $F$ generated by four elements $x_1,x_2, y_1,y_2$.
	The basic commutators $[x,y]$ form a basis for the free abelian group $\gamma_2(F)/\gamma_3(F)$.
	We fix the following notations:
	$$[y_1,x_1]=z_1, [y_1,x_2]=z_2, [y_2,x_1]z_2^{-1}=w_2, [x_2,x_1]=w_1, [y_2,x_2]=w_3, [y_2,y_1]=w_4,$$
	and
	$$
	 [z_1, x_i]=v_i, [z_1, y_i]=u_i, i=1,2.$$ To prove our result, we show that $G^\star$ is isomorphic to a certain quotient group of $F$.
	
	Let $H_1= F/\langle \gamma_3(F), F^p, w_1,w_2, w_3, w_4\rangle$. Then
	\begin{eqnarray*}
		H_1 &\cong& \langle  y_1,y_2,x_1,x_2 \mid [y_1,x_1]=z_1,  [y_1,x_2]=[y_2,x_1]=z_2, x_i^p=y_i^p=1 \rangle\\
		&\cong&  \hfpp.
	\end{eqnarray*}
	is a group of order $p^6$.
	
We have the following identities in $F$ modulo $\langle \gamma_4(F),  [w_j, x_i],  [w_j, y_i], ~i=1,2, ~1\leq j \leq 4 \rangle$.
	\begin{eqnarray*}
		&&[x_1^{-1}, y_1^{-1}, x_2]=[z_1,x_2]^{-1},~~[x_1^{-1}, y_1^{-1}, y_2]=[z_1,y_2]^{-1}, \\
		&&[x_1^{-1}, y_2^{-1}, x_2]=[z_2,x_2]^{-1},~~[x_2^{-1}, y_1^{-1}, y_2]=[z_2,y_2]^{-1}, \\
		&&[y_1, x_2^{-1}, x_1^{-1}]=[z_2,x_1].
	\end{eqnarray*}

	Observe that,  using Hall-Witt identity we  have the following relations in $F$ modulo $\langle \gamma_4(F),  [w_j, x_i],  [w_j, y_i], ~i=1,2, ~1\leq j \leq 4 \rangle$.
	\begin{eqnarray}\label{HallWitt}
		&&[z_2,x_1]=[z_1,x_2], ~~~[z_2,y_1]=[z_1,y_2], ~~~[z_2,y_2]=[z_2,x_2]=1.
	\end{eqnarray}
	Now consider the group
	\[
	H_2=F/\langle \gamma_4(F), F^p, w_1,w_2, w_3, w_4, [z_1,x_2], [z_1,y_i],  i=1,2 \rangle.
	\]
	By (\ref{HallWitt}), we have
	\[
	H_2 \cong \langle  x_1, x_2, y_1, y_2 \mid [y_1,x_1]=z_1, [y_1,x_2]=[y_2,x_1]=z_2, [z_1,x_1]=v_1, x_i^p=1 \rangle.
	\]
	and $H_2/\langle v_1 \rangle \cong H_1$. Hence $H_2$ is of order $p^7$. We proceed further step by step considering the groups
	\[
	H_3=F/\langle \gamma_4(F), F^p,  w_1,w_2, w_3, w_4,  [z_1,y_i],  i=1,2 \rangle,
	\]
	\[
	H_4=F/\langle \gamma_4(F), F^p,  w_1,w_2, w_3, w_4, [z_1,y_2] \rangle,
	\]
	\[
	H_5=F/\langle \gamma_4(F), F^p,  w_1,w_2, w_3, w_4\rangle.
	\]
	Using the identities in (\ref{HallWitt}), we see  that
	\begin{eqnarray*}
		H_5 &\cong& \langle  x_1, x_2, y_1, y_2 \mid [y_1,x_1]=z_1, [y_1,x_2]=[y_2,x_1]=z_2, [z_1,x_1]=v_1, [z_1,y_1]=u_1,\\
		&&\hspace{5 cm} [z_2,x_1]=[z_1,x_2]=v_2, [z_2,y_1]=[z_1,y_2]=u_2, x_i^p=1 \rangle.
	\end{eqnarray*}
	and $H_5/\langle u_2 \rangle \cong H_4, H_4/\langle u_1 \rangle \cong H_3, H_3/\langle v_2 \rangle \cong H_2$.
	Hence $H_5$ is of order $p^{10}$.

Now consider the group
 \[
K_1=F/\langle \gamma_4(F), F^p, w_1,w_3, w_4, [w_2,x_i], [w_2,y_i], i=1,2 \rangle.
\]
By (\ref{HallWitt}), we have  $[z_2,x_1]=[z_1,x_2]=v_2, ~[z_2,y_1]=[z_1,y_2]=u_2, ~[z_2,y_2]=[z_2,x_2]=1.$
So
\begin{eqnarray*}
K_1 &\cong&  \langle  y_1, y_2, x_1,x_2 \mid [y_1,x_1]=z_1, [y_1,x_2]=z_2, [y_2,x_1]z_2^{-1}=w_2, [z_1,x_1]=v_1, [z_1,y_1]=u_1,\\
&&\hspace{5 cm} [z_2,x_1]=[z_1,x_2]=v_2, [z_2, y_1]=[z_1,y_2]=u_2, x_i^p=1 \rangle.
\end{eqnarray*}
Therefore $K_1/\langle w_2 \rangle \cong H_5$. Hence $K_1$ is of order $p^{11}$. In the similar way, step by step, we consider the groups
\[
K_2=F/\langle \gamma_4(F), F^p,w_3, w_4, [w_j,x_i], [w_j,y_i], ~i=1,2, ~j=1,2 \rangle,
\]
\[
K_3=F/\langle \gamma_4(F), F^p, w_4, [w_j,x_i], [w_j,y_i], ~i=1,2, ~j=1,2,3 \rangle,
\]
\[
K_4=F/\langle \gamma_4(F), F^p, [w_j,x_i], [w_j,y_i], ~i=1,2, ~j=1,2,3,4 \rangle,
\]
and  we see that $K_2/\langle w_1 \rangle \cong K_1, ~K_3/\langle w_3 \rangle \cong K_2, ~K_4/\langle w_4 \rangle \cong K_3$. Finally, we see that $K_4 \cong G^\star$.
Hence $G^\star$ is of order $p^{14}$.
\end{proof}

In the following corollary, we point out the cardinality of a specific group that appeared in the above proof. We require it for later.

\begin{corollary}
	Consider the following  nilpotent class three group:
	\begin{eqnarray}\label{tildeG}
		\tilde{G}&=& \langle   x_i,y_i , i=1,2 \mid [y_1,x_1]=z_1,  [y_1,x_2]=[y_2,x_1]=z_2,  [z_1, x_i]=v_i, [z_1, y_i]=u_i,  \nonumber \\
		&&  \hspace{3cm}  [z_2,x_1]=[z_1, x_2],   [z_2,y_1]=[z_1, y_2],  x_i^p=y_i^p=z_i^p=1 \rangle.
	\end{eqnarray}	
	The group $\tilde{G}$ is of order $p^{10}$. 	
\end{corollary}
\begin{proof}
	Let $H_5$ be the group that appeared in the proof of Theorem~\ref{representation-group}, where we also proved that $|H_5| = p^{10}$. The result follows because $\tilde{G} \cong H_5$.
\end{proof}

\subsection{$\mathbb{H}_{3}\Big(\frac{\mathbb{F}_p[t]}{(t^2-k)}\Big)$}
\label{sec:representation-group-field} In this section, we give a representation group of $\mathbb{H}_{3}\Big(\frac{\mathbb{F}_p[t]}{(t^2-k)}\Big)$.

\begin{thm}
A representation  group of $G = \mathbb{H}_{3}\Big(\frac{\mathbb{F}_p[t]}{(t^2-k)}\Big)$ is the following nilpotency class three group:
 \begin{eqnarray*}
 	G^\star&=&\langle  x_i, y_i, i=1,2  \mid [y_1,x_1]=z_1,[y_1,x_2]=z_2, [y_2,x_1]z_2^{-1}=w_2, [y_2,x_2]z_1^{-k}=w_3,  \\
 	&&[z_1, y_i]=u_i, [z_1, x_i]=v_i , [z_2,x_1]=[z_1, x_2],   [z_2,y_1]=[z_1, y_2],  [z_2,y_2]=[z_1,y_1]^k,    \\
 	&&[z_2,x_2]=[z_1,x_1]^k, [x_2,x_1]=w_1,  [y_2,y_1]=w_4,
 	x_i^p=y_i^p=w_i^p=1\rangle.
 \end{eqnarray*}
\end{thm}
\begin{proof}  Consider the group
	\begin{eqnarray*}
		H &=&\langle   x_i, y_i, i=1,2  \mid [y_1,x_1]=z_1, [y_2,x_2]=z_1^k, [y_1,x_2]=[y_2,x_1]=z_2,  [z_1, y_i]=u_i\\
		&&[z_1, x_i]=v_i,  [z_2,x_1]=[z_1, x_2],  [z_2,y_1]=[z_1, y_2],   [z_2,y_2]=[z_1,y_1]^k,   \\
		&&  [z_2,x_2]=[z_1,x_1]^k, x_i^p=y_i^p=w_i^p=1 \rangle.
	\end{eqnarray*}
By using the method similar to the proof of Theorem \ref{representation-group}, we obtain that $H$ is of order $p^{10}$ and $G^\star$ is of order $p^{14}$. Hence $G^\star$ is a representation group of $\mathbb{H}_{3}\Big(\frac{\mathbb{F}_p[t]}{(t^2-k)}\Big)$.
\end{proof}

\subsection{$\mathbb{H}_{3}(\mathbb{Z}/p\mathbb{Z}\times \mathbb{Z}/p\mathbb{Z})$}
\begin{thm} A representation  group of $G = \mathbb{H}_{3}(\mathbb{Z}/p\mathbb{Z}\times \mathbb{Z}/p\mathbb{Z})$ is the following nilpotency class three group:
\begin{eqnarray*}
	G^\star &=&\langle x_1,x_2,y_1,y_2 \mid [x_1,x_2]=z_1, [y_1,y_2]=z_2, [z_1, x_1]=t_1, [z_1, x_2]=t_2, [z_2, y_1]=t_3,\\
	&& [z_2, y_2]=t_4, [x_1,y_1]=t_5, [x_1,y_2]=t_6,  [x_2,y_1]=t_7,  [x_2,y_2]=t_8, \\
	&&\hspace{8cm} x_i^p=y_i^p=z_1^p=z_2^p=1, i=1,2 \rangle.
\end{eqnarray*}

\end{thm}
\begin{proof}
We have $$
\Ho^2(G, \mathbb C^\times) \cong  (\mathbb{Z}/p\mathbb{Z})^8$$
and
\begin{eqnarray*}
G^\star & \cong&  \Big(\langle x_1,x_2\mid [x_1,x_2]=z_1, [z_1, x_1]=t_1, [z_1, x_2]=t_2 \rangle \times \langle  t_5,t_6,t_7,t_8\rangle\Big) \\
&&\rtimes \langle y_1,y_2 \mid   [y_1,y_2]=z_2, [z_2, y_1]=t_3,  [z_2, y_2]=t_4\rangle\\
&& \cong \big(H' \times ( \mathbb Z/p\mathbb Z)^4\big)  \rtimes K',
\end{eqnarray*}
where $H'\cong K'$ is of order $p^5$.
 Thus $G^\star$ is  a semidirect product of a normal subgroup of order $p^9$ and a subgroup of order $p^5$.
 So $G^\star$ of order $p^{14}$ such that the  following sequence
$$
1 \to \langle t_i\rangle_{i=1}^{8}  \cong \Ho^2(G, \mathbb C^\times) \to G^\star \to G \to 1
$$
 is exact.
Hence the result follows.
\end{proof}



\section{Description of $2$-cocycles of $\mathbb{H}_{3}(R)$}\label{section-cocycle-description}

In this section we describe the $2$-cocycles of $\mathbb{H}_3(R).$
\subsection{ $\hfpp$}

We start with a proof of Theorem \ref{2cocycle}.
\begin{proof}[Proof of Theorem \ref{2cocycle}]
	Let $G=\hfpp$. We show that $\Ho^2(G, \bC^\times)$ is the direct product of two groups which we compute explicitly. In order to do this, consider the central subgroup $Z=\langle z_1 \rangle$ of $G$.
	Then
	$$H=G/Z \cong \langle  y_1, y_2,x_1, x_2\mid   [y_1,x_2]=[y_2,x_1]=z_2, x_i^p=y_i^p=1   \rangle$$
	 is an extra-special $p$-group of order $p^5$.  From \eqref{exact-sequence}, we obtain the exact sequence
	\begin{eqnarray*}
		&&  1 \to \Hom(Z, \bC^\times)   \xrightarrow{\tra} \Ho^2(H, \bC^\times)   \xrightarrow{\inf}  \Ho^2(G, \bC^\times)  \xrightarrow{\chi}  \Hom(G/G' \otimes Z,  \bC^\times).
	\end{eqnarray*}
	Since
	$\Ho^2(G,  \bC^\times)\cong (\bZ/p\bZ)^8, \Ho^2(H,  \bC^\times)\cong (\bZ/p\bZ)^5,$
	$\mathrm{Im}(\inf)\cong \frac{\Ho^2(H, \bC^\times)}{\mathrm{Im}(\tra)}$
	and $\Hom(G/G' \otimes Z,  \bC^\times) \cong (\bZ/p\bZ)^4$
	we have, by Proposition \ref{exponent-p},
		\begin{eqnarray}\label{split}
	\Ho^2(G,  \bC^\times)\cong \mathrm{Im}\big( \inf) \times \mu\big(\Hom(G/G' \otimes Z,  \bC^\times)\big),
		\end{eqnarray}
for a section $\mu:\Hom(G/G' \otimes Z,  \bC^\times)  \to  \Ho^2(G, \bC^\times)$  which is an injective homomorphism($\mu$ will be defined in Step 2 below).
	 We describe  $\mathrm{Im}(\inf)$ and $ \mu\big(\Hom(G/G' \otimes Z,  \bC^\times)\big)$ in the following two steps. For simplification of notations, let $g = x_1^{a_1}x_2^{a_2}y_1^{b_1}y_2^{b_2}z_1^{c_1}z_2^{c_2}$, $g' =  x_1^{a_1'}x_2^{a_2'}y_1^{b_1'}y_2^{b_2'}z_1^{c_1'}z_2^{c_2'}$ such that $g$ and $g'$ are arbitrary elements of $G$ and let $N = \langle Z, z_2 \rangle$
	 be the subgroup generated by $Z$ and $z_2$. \\

\noindent \textbf{Step 1:} To describe $\mathrm{Im}(\inf)$ we proceed as follows.
Consider the exact sequence
\begin{eqnarray*}
		&&  1 \to \Hom(\langle z_2 \rangle, \bC^\times)   \xrightarrow{\tra} \Ho^2(H/\langle z_2 \rangle, \bC^\times)   \xrightarrow{\inf}  \Ho^2(H, \bC^\times).
	\end{eqnarray*}
	 Since $\Ho^2(H, \bC^\times)\cong (\bZ/p\bZ)^5$ and $H/\langle z_2 \rangle \cong (\bZ/p\bZ)^4,$ the map
$$\inf:  \Ho^2(H/\langle z_2 \rangle,\bC^\times)\to \Ho^2(H,\bC^\times)$$ is surjective. Hence every  $2$-cocycle $\beta$ of $H$ is of the form $[\beta]=\inf([\delta]), \delta \in \Ho^2(H/\langle z_2 \rangle, \bC^\times)   $. The $2$-cocycles of $H/\langle z_2 \rangle$, being an elementary abelian group, are well known. Therefore
	\begin{eqnarray*}
		\beta\big(gZ, g'Z\big)=\delta\big(gN, g'N\big)
		=\mu_1^{a_2a_1'} \lambda_1^{b_1a_1'}\lambda_2^{b_2a_1'}\lambda_3^{b_1a_2'}\mu_3^{b_2a_2'} \mu_4^{b_2 b_1'},
	\end{eqnarray*}
where $\mu_j, ~j \in \{ 1, 3, 4\}$ and $\lambda_\ell, ~\ell \in \{1, 2, 3\}$ are scalars whose $p$-th power is one.

	Now recall (\ref{split}) and define $[\alpha_1] =\inf([\beta]) \in  \Ho^2(G, \bC^\times)$ for $\beta \in \Ho^2(H, \bC^\times)$.  Then
	\begin{eqnarray*}
		 \alpha_1\big(g, g' \big)=\beta\big(gZ, g'Z\big)
		=\mu_1^{a_2a_1'} \lambda_1^{b_1a_1'}\lambda_2^{b_2a_1'}\lambda_3^{b_1a_2'}\mu_3^{b_2a_2'} \mu_4^{b_2 b_1'}.
	\end{eqnarray*}
	For $i\in \{ 1,2 \}$, define maps $f_i:G \to \bC^\times$ by $$f_1(g)=\lambda_1^{c_1},~~f_2(g)=\lambda_3^{c_2},
	 \text{ for } g=x_1^{a_1}x_2^{a_2}y_1^{b_1}y_2^{b_2}z_1^{c_1}z_2^{c_2} \in G.$$
This shows that
	$\lambda_1^{b_1a_1'}, \lambda_3^{b_2a_1'+b_1a_2'}$ are coboundaries in $G$. Hence up to cohomologous,
	\begin{eqnarray}\label{cocycle1}
	 \alpha_1\big(g, g'\big)=\mu_1^{a_2a_1'}({\lambda_2\lambda_3^{-1}})^{b_2a_1'}\mu_3^{b_2a_2'} \mu_4^{b_2 b_1'}
	=\mu_1^{a_2a_1'} \mu_2^{b_2a_1'}\mu_3^{b_2a_2'} \mu_4^{b_2 b_1'}.
	\end{eqnarray}
	Therefore, in (\ref{split}), $\mathrm{Im}(\inf)$  consists of $[\alpha_1]$ such that $\alpha_1$ is a $2$-cocycle of the above form. \\

\noindent	\textbf{Step 2:} Consider the group $\tilde{G}$ from (\ref{tildeG}).
	We have the central exact sequence
	$$1 \to \langle u_1,u_2,v_1,v_2 \rangle \xrightarrow{i} \tilde{G} \xrightarrow{\pi} G \to 1.$$
	 Observe that in $\tilde{G}$, we have
	\begin{eqnarray*}
		&&[y_1^i, x_1^j]=z_1^{ij}v_1^{i{j \choose 2}}u_1^{j {i \choose 2}}, ~ [y_1^i, x_2^j]=z_2^{ij}u_2^{j {i \choose 2}},~[y_2^i, x_1^j]=z_2^{ij}v_2^{i{j \choose 2}}.
	\end{eqnarray*}
For $g=x_1^{a_1}x_2^{a_2}y_1^{b_1}y_2^{b_2}z_1^{c_1}z_2^{c_2}, g'=x_1^{a_1'}x_2^{a_2'}y_1^{b_1'}y_2^{b_2'}z_1^{c_1'}z_2^{c_2'}$ in $G$,
 $$gg'=x_1^{a_1+a_1'}x_2^{a_2+a_2'}y_1^{b_1+b_1'}y_2^{b_2+b_2'}z_1^{c_1+c_1'+b_1a_1'}z_2^{c_2+c_2'+b_1a_2'+b_2a_1'}.$$

	We denote the numbers $b_1{a_1' \choose 2}+a_1'c_1,$  $b_2{a_1' \choose 2}+a_1'c_2+a_2'c_1,$ $a_1' {b_1\choose 2}-b_1c_1',$ $a_2'{b_1\choose 2}-b_2c_1'-b_1c_2'$ by $k_5, k_6, k_7, k_8$ respectively.
Define a section $s:G \to \tilde{G}$ by $$s(x_1^{a_1}x_2^{a_2}y_1^{b_1}y_2^{b_2}z_1^{c_1}z_2^{c_2})=
	x_1^{a_1}x_2^{a_2}z_1^{c_1}z_2^{c_2}y_1^{b_1}y_2^{b_2}
	 \in \tilde{G}.$$ Then,
\begin{eqnarray*}
		s(g)s(g')&=& (x_1^{a_1}x_2^{a_2}z_1^{c_1}z_2^{c_2}y_1^{b_1}y_2^{b_2})(x_1^{a_1'}x_2^{a_2'}z_1^{c_1'}z_2^{c_2'}y_1^{b_1'}y_2^{b_2'})\\
		&=&x_1^{a_1+a_1'}x_2^{a_2+a_2'}z_1^{c_1+c_1'+b_1a_1'}z_2^{c_2+c_2'+b_1a_2'+b_2a_1'}y_1^{b_1+b_1'}y_2^{b_2+b_2'}v_1^{k_5}v_2^{k_6+b_1a_1'a_2'}\\
		&&  u_1^{k_7-b_1^2a_1'}  u_2^{k_8-b_1(b_1a_2'+b_2a_1')},\\
		s(g)s(g')s(gg')^{-1}&=&v_1^{k_5}v_2^{k_6+b_1a_1'a_2'}  u_1^{k_7-b_1^2a_1'}  u_2^{k_8-b_1(b_1a_2'+b_2a_1')}
	\end{eqnarray*}

	Consider the exact sequence
	\begin{eqnarray*}
		&&  1 \to \Hom(\langle u_1,u_2,v_1,v_2 \rangle, \bC^\times)   \xrightarrow{\tra} \Ho^2(G, \bC^\times)   \xrightarrow{\inf}  \Ho^2(\tilde{G}, \bC^\times).
	\end{eqnarray*}
	Let $f \in \Hom(\langle u_1,u_2,v_1,v_2\rangle, \bC^\times)$  and  $f(v_1), f(v_2), f(u_1),f(u_2)$ be denoted by  $\mu_{5}, \mu_{6}, \mu_{7},\mu_{8} \in \bC^\times$ respectively. Then
	$[\alpha_2]=\tra(f) \in \Ho^2(G,\bC^\times)$ is defined by
	\begin{eqnarray*}
	\hspace{.2cm} \alpha_2\big(g, g' \big) &&= f(s(g)s(g')s(gg')^{-1})\nonumber \\
	&&=  f\big(v_1^{k_5}v_2^{k_6+b_1a_1'a_2'}  u_1^{k_7-b_1^2a_1'}  u_2^{k_8-b_1(b_1a_2'+b_2a_1')}\big)\nonumber \\
	&&=\mu_5^{k_5}\mu_6^{k_6+b_1a_1'a_2'}\mu_7^{k_7-b_1^2a_1'}  \mu_8^{k_8-b_1(b_1a_2'+b_2a_1')}
	\end{eqnarray*}	
	For $i\in \{1,2,3\}$, define the maps
	$f_i:G \to \bC^\times$ by $$f_1(g)=\mu_6^{c_1},~~f_2(g)=\mu_7^{c_1},~~ f_3(g)=\mu_8^{c_2} \text{ for } g=x_1^{a_1}x_2^{a_2}y_1^{b_1}y_2^{b_2}z_1^{c_1}z_2^{c_2} \in G.$$ This gives that $\mu_6^{b_1a_1'a_2'}, \mu_7^{-b_1^2a_1'},\mu_8^{(b_1a_2'+b_2a_1')}$ are coboundaries in $G$.
	Therefore up to cohomologous we have
	\begin{eqnarray}\label{cocycle2}
		\alpha_2\big(g, g'\big) =\mu_5^{k_5}\mu_6^{k_6}\mu_7^{k_7}  \mu_8^{k_8}.
		\end{eqnarray}
	Consider the subgroup $S$ of $\Ho^2(G,\mathbb C^\times)$ consisting of $[\alpha_2]$ such that $\alpha_2$ of the  form (\ref{cocycle2}). Now
	\begin{eqnarray*}
\chi([\alpha_2])\big(x_1^{a_1}x_2^{a_2}y_1^{b_1}y_2^{b_2}G' \otimes z_1^{c_1}\big)&=&\alpha_2\big(x_1^{a_1}x_2^{a_2}y_1^{b_1}y_2^{b_2}, z_1^{c_1}\big)\alpha_2\big(z_1^{c_1}, x_1^{a_1}x_2^{a_2}y_1^{b_1}y_2^{b_2}\big)^{-1}\\
		&=&  \mu_5^{-a_1c_1}\mu_6^{-a_2c_1} \mu_7^{-b_1c_1}  \mu_8^{-b_2c_1}.
	\end{eqnarray*}
	Define a section $\mu: \Hom(G/G' \otimes Z, \mathbb C^\times)\to S$ by
		\begin{eqnarray*}
\mu(f)\big(g, g'\big)=f(x_1\otimes z_1)^{-k_5}f(x_2\otimes z_1)^{-k_6}f(y_1\otimes z_1)^{-k_7} f(y_2\otimes z_1)^{-k_8}.
	\end{eqnarray*}
It is easy to check that $\mu$ is a homomorphism and  $\chi |_{S}\circ \mu=\mathrm{id}|_{\Hom(G/G' \otimes Z, \mathbb C^\times)}$ and $\mu \circ \chi |_{S}=\mathrm{id}|_{S}$.
In particular,
\[
\mu\big(\Hom(G/G' \otimes Z, \mathbb C^\times)\big)\cong S.
\]

We are now in a position to complete the proof.
	Using (\ref{cocycle1}) and  (\ref{cocycle2}) we define
	$\alpha=\alpha_1\alpha_2$ and so we have
	\begin{eqnarray}\label{cocycle7}
	&& \hspace{.2cm}   \alpha\big(x_1^{a_1}x_2^{a_2}y_1^{b_1}y_2^{b_2}z_1^{c_1}z_2^{c_2}, x_1^{a_1'}x_2^{a_2'}y_1^{b_1'}y_2^{b_2'}z_1^{c_1'}z_2^{c_2'}\big)=\mu_1^{a_2a_1'} \mu_2^{b_2a_1'}\mu_3^{b_2a_2'} \mu_4^{b_2 b_1'}  \mu_5^{k_5}\mu_6^{k_6}\mu_7^{k_7}  \mu_8^{k_8}.
	\end{eqnarray}
Therefore the result follows.
\end{proof}

\subsection{ $\mathbb{H}_3\Big(\frac{\mathbb{F}_p[t]}{(t^2-k)}\Big)$}

We denote the numbers  $b_1{a_1' \choose 2}+kb_1{a_2' \choose 2}+ kb_2a_1'a_2'+a_1'c_1+ka_2'c_2, ~b_2{a_1' \choose 2}+kb_2{a_2'\choose 2}+b_1a_1'a_2'+a_1'c_2+a_2'c_1, ~a_1' {b_1\choose 2}+ka_1'{b_2 \choose 2}-b_1^2a_1'-b_1c_1'-kb_2c_2'$ and $a_2'{b_1\choose 2}+ka_2'{b_2\choose 2}-kb_2^2a_2'-b_2c_1'-b_1c_2'$ by $q_5, q_6, q_7$ and $q_8$ respectively.
\begin{thm}
Upto cohomologous, every $2$-cocycle $\alpha$ of $\mathbb{H}_{3}\Big(\frac{\mathbb{F}_p[t]}{(t^2-k)}\Big)$ is of the following form:
\begin{eqnarray*}
&&\alpha\big(x_1^{a_1}x_2^{a_2}y_1^{b_1}y_2^{b_2}z_1^{c_1}z_2^{c_2}, x_1^{a_1'}x_2^{a_2'}y_1^{b_1'}y_2^{b_2'}z_1^{c_1'}z_2^{c_2'}\big)=\mu_1^{a_2a_1'} \mu_2^{b_2a_1'}\mu_3^{b_2a_2'} \mu_4^{b_2 b_1'}  \mu_5^{q_5} \mu_6^{q_6} \mu_7^{q_7} \mu_8^{q_8},
\end{eqnarray*}
 where  $\mu_i \in \bC^\times$ such that $\mu_i^p=1$.	
\end{thm}
\begin{proof}
The proof of this result goes on the same lines as the proof of Theorem \ref{2cocycle}, using the representation group of $\mathbb{H}_3\Big(\frac{\mathbb{F}_p[t]}{(t^2-k)}\Big)$ that appeared in Section~\ref{sec:representation-group-field}.
\end{proof}

\subsection{$\mathbb{H}_3(\mathbb Z/p\mathbb Z \times \mathbb Z/p\mathbb Z)$}

In the following result, the element $x_1^{a_1}x_2^{a_2}y_1^{b_1}y_2^{b_2}z_1^{c_1}z_2^{c_2}$ denotes the element
\[
\begin{pmatrix}
1 & (b_1,b_2) & (c_1,c_2)\\
0 & 1& (a_1,a_2)\\
0 & 0& 1
 \end{pmatrix}
 \in \mathbb{H}_3(\bZ/p\bZ \times \bZ/p\bZ).
\]
\begin{thm}
Up to cohomologous, every $2$-cocycle $\alpha$ of $\mathbb{H}_{3}\big(\mathbb Z/p\mathbb Z \times \mathbb Z/p\mathbb Z \big)$ is of the following form:
\begin{eqnarray*}
&&\alpha \big(x_1^{a_1}x_2^{a_2}y_1^{b_1}y_2^{b_2}z_1^{c_1}z_2^{c_2}, x_1^{a_1'}x_2^{a_2'}y_1^{b_1'}y_2^{b_2'}z_1^{c_1'}z_2^{c_2'}\big)
=\lambda_1^{c_1'a_1+b_1{a_1'\choose 2}+a_1b_1a_1'}\\
&&\hspace{2.6 cm} \lambda_2^{c_1'b_1+a_1'{b_1\choose 2}} \lambda_3^{c_{2}'a_{2}+b_{2}{a_{2}'\choose 2}+a_{2}b_{2}a_{2}'}\lambda_4^{c_{2}'b_{2}+a_{2}'{b_{2}\choose 2}}
\lambda_5^{a_1'a_{2}}\lambda_6^{a_1'b_{2}}\lambda_7^{b_1'a_2}\lambda_8^{b_1'b_2},
\end{eqnarray*}

where $\lambda_i \in \mathbb C^\times$ such that $\lambda_i^p=1$.
\end{thm}
\begin{proof}
	Since $\mathbb{H}_{3}(\mathbb Z/p\mathbb Z \times \mathbb Z/p\mathbb Z) \cong \mathbb{H}_{3}(\mathbb{Z}/p\mathbb{Z})\times \mathbb{H}_{3}(\mathbb{Z}/p\mathbb{Z})$,
result  follows from \cite[Theorem 9.6]{MA}.
\end{proof}

\subsection{$\mathbb{H}_{3}\big(\mathbb Z/p^2\mathbb Z  \big)$} In the following result, the element $x^ay^bz^c$ denotes the element
\[
\begin{pmatrix}
1 & b & c\\
0 & 1& a\\
0 & 0& 1
\end{pmatrix}
\in \mathbb{H}_3(\bZ/p^2\bZ).
\]
\begin{thm}
Upto cohomologous, every $2$-cocycle $\alpha$ of   $\mathbb H_3(\mathbb Z/p^2\mathbb Z)$ is of the following form:
\[
\alpha\big(x^{a}y^{b}z^{c}, x^{a'}y^{b'}z^{c'}\big)=\lambda_1^{c'a+b{a'\choose 2}+aba'}\lambda_2^{c'b+a'{b\choose 2}}, \,\, \lambda_1^{p^2}= \lambda_2^{p^2}=1.
\]
\end{thm}
\begin{proof}
The cocycle description  follows from \cite[Lemma 2.2]{PS}.
\end{proof}

\section{Projective representations of $\mathbb H_3(R)$}\label{Non-degenerate-cocycles}

In this section we give a construction of all projective irreducible representations of   $\hfpp$ and $\hzpp$. We also classify their non-degenerate cocycles. In particular,  Theorem~\ref{thm:number-non-degenerate} follows from this section.

\subsection{Projective representations of $\hfpp$}
\label{section:projection-function-field}

By Theorem~\ref{representation-group}, a representation group of $G = \hfpp$ is given by

\begin{eqnarray*}
	G^\star&=&\langle  x_i,y_i , i=1,2 \mid [y_1,x_1]=z_1,  [y_1,x_2]=z_2, [x_2,x_1]=w_1,[y_2,x_1]z_2^{-1}=w_2,
	\\
	&&  [y_2,x_2]=w_3,[y_2,y_1]=w_4,  [z_1, x_i]=v_i,  [z_1, y_i]=u_i, [z_2,x_1]=[z_1, x_2],   \\
	&&  [z_2,y_1]=[z_1, y_2], x_i^p=y_i^p=z_i^p=w_j^p=1, 1 \leq j \leq 4  \rangle.
\end{eqnarray*}
Consider the normal subgroup $N =\langle z_i, u_i,v_i, w_j,  i=1,2, 1\leq j \leq 4\rangle \cong (\mathbb Z/p\mathbb Z)^{10}$ of $G^\star$.
Any one dimensional ordinary representation, say $\chi$, of $N$ is given by
$$\chi(w_j)=\mu_j, 1\leq j \leq 4, \chi(u_1)=\mu_7, \chi(u_2)=\mu_6,  \chi(v_1)=\mu_5, \chi(v_2)=\mu_8, \chi(z_1) = \mu_9, \chi(z_2) = \mu_{10},$$
where $\mu_i \in \mathbb C^\times$ such that $\mu_i^p = 1$ for all $i$. 	If $\mu_6 \neq 1$, assume that $\mu_i=\mu_6^{r_i}$ for $1 \leq i \leq 8$ with $r_6 = 1$
and,  if $\mu_8 \neq 1$ assume that $\mu_i=\mu_8^{t_i}$ for $1 \leq i \leq 7$ with $t_8 = 1$. \\

Our first step is to determine the stabilizer $I_{G^\star}(\chi)$ of the character $\chi$. An element $g = x_1^{a_1}x_2^{a_2}y_1^{b_1}y_2^{b_2}n$ for  $n\in N$ satisfies $g \in I_{G^\star}(\chi)$ if and only if
$\chi([x_1^{a_1}x_2^{a_2}y_1^{b_1}y_2^{b_2}, z_1])=1$ and  $\chi([x_1^{a_1}x_2^{a_2}y_1^{b_1}y_2^{b_2}, z_2])=1$.  From the definition of $G^\star$, we observe that
\begin{eqnarray}\label{commutator}
	&&[x_1^{a_1}x_2^{a_2}y_1^{b_1}y_2^{b_2}z_1^{c_1}z_2^{c_2}, x_1^{a_1'}x_2^{a_2'}y_1^{b_1'}y_2^{b_2'}z_1^{c_1'}z_2^{c_2'}] = z_1^{a_1'b_1-a_1b_1'}z_2^{a_2'b_1-a_2b_1'+a_1'b_2-a_1b_2'}w_1^{a_1'a_2-a_1a_2'}\nonumber \\
	&&\hspace{3.6cm} w_2^{a_1'b_2-a_1b_2'}w_3^{a_2'b_2-a_2b_2'}w_4^{b_1'b_2-b_1b_2'}u_1^{a_1'{b_1\choose 2}-a_1{b_1'\choose 2}+b_1b_1'(a_1'-a_1)+(b_1'c_1-b_1c_1')}\nonumber \\
	&& \hspace{1.2cm} u_2^{a_2'{b_1\choose 2}-a_2{b_1'\choose 2}+b_1b_1'(a_2'-a_2)-a_1b_1'(b_2'+b_2)         +b_1b_2'(a_1'-a_1)+a_1'b_2(b_1+b_1')+(b_1'c_2-b_1c_2')+(b_2'c_1-b_2c_1')}\nonumber\\
	&&\hspace{1.2cm} v_1^{b_1{a_1'\choose 2}-b_1'{a_1\choose 2}+a_1'c_1-a_1c_1'}v_2^{b_2{a_1'\choose 2}-b_2'{a_1\choose 2}+(a_1'a_2'b_1-a_1a_2b_1')+(a_1'c_2-a_1c_2')+(a_2'c_1-a_2c_1')}.
\end{eqnarray}
Therefore $g \in I_{G^\star}(\chi)$ if and only if
\begin{eqnarray}\label{stabilizer}
&& \chi(v_1)^{a_1} \chi(v_2)^{a_2}\chi(u_1)^{b_1}\chi(u_2)^{b_2}=1 \text { and } \chi(v_2)^{a_1}\chi(u_2)^{b_1}=1.
\end{eqnarray}
Our next goal is to describe  $\mathrm{Irr}(I_{G^\star}(\chi) \mid \chi)$. We note that $I_{G^\star}(\chi)/N$ is abelian and $|I_{G^\star}(\chi)/N| \leq p^4$. Therefore, by Clifford theory, all representations in $\mathrm{Irr}(I_{G^\star}(\chi) \mid \chi)$ for a fixed $\chi$ will have the
same dimension. We now consider various cases: \\


$(i)$ Assume $\mu_6=\mu_8=1$.
Then by (\ref{stabilizer}),  $g \in I_{G^\star}(\chi)$ if and only if $\chi(v_1)^{a_1}\chi(u_1)^{b_1}=1$.
Hence $\left \vert \frac{ I_{G^\star}(\chi)}{N} \right \vert \in  \{p^3, p^4\}.$ We consider these cases separately.

Suppose { $\left \vert \frac{ I_{G^\star}(\chi)}{N} \right \vert = p^4$.} For this case, $I_{G^\star}(\chi) = G^\star$. by Lemma~\ref{lemma:quotient-group-bijection}, representations are determined by certain projective representations of the abelian group $G^\star/N$. Since $\vert G^\star/N \vert = p^4$, any irreducible representation of $\irr(G^\star \mid \chi)$ will be of dimension either $p$ or $p^2.$

Next suppose { $\left \vert \frac{ I_{G^\star}(\chi)}{N}\right \vert = p^3$.} Again by Lemma~\ref{lemma:quotient-group-bijection}, representations are determined by certain projecive representative representations of the abelian group $I_{G^\star}(\chi)/N$. Since $\vert G^\star/N \vert = p^4$, any irreducible representation of $\irr(G^\star \mid \chi)$ in this case will too be of dimension either $p$ or $p^2.$ \\

$(ii)$ Assume  that $\mu_6=1, \mu_8\neq1$.  Then by (\ref{stabilizer}), $g = x_1^{a_1}x_2^{a_2}y_1^{b_1}y_2^{b_2}n \in I_{G^\star} (\chi)$ if and only if  $a_1=0$ and  $a_2=-t_7b_1$. Therefore $|I_{G^\star}(\chi)/N| = p^2$ in this case. By Lemma \ref{lemma:dimension-of-abelian-representations} and the discussion before it, an irreducible representation of $\irr(I_{G^\star}(\chi) \mid \chi)$ will be of dimension one or $p$. Further it will be one dimensional if and only if $\chi$ extends to $I_{G^\star}(\chi)$. By Lemma~\ref{lemma:extension-condition}, this holds if and only if $[I_{G^\star}(\chi), I_{G^\star}(\chi)] \subseteq \mathrm{Ker}(\chi)$, that is
\begin{eqnarray*}
&& \chi\big([x_2^{-t_7b_1}y_1^{b_1}y_2^{b_2}z_1^{c_1}z_2^{c_2}, x_2^{-t_7b_1'}y_1^{b_1'}y_2^{b_2'}z_1^{c_1'}z_2^{c_2'}]\big)=1.
\end{eqnarray*}
By using (\ref{commutator}), this is equivalent to
\begin{eqnarray*}
 \chi(v_2)^{(t_3t_7-t_4)(b_1b_2'-b_1'b_2)}=1,
\end{eqnarray*}
where $b_1, b_2, b_1', b_2'$ are arbitrary. So we must have $t_4=t_3t_7$.

This discussion altogether implies that for $t_4 = t_3t_7$, $\irr(G^\star \mid \chi)$ consists of $p^2$ dimensional irreducible representations and for $t_4\neq t_3t_7$, all representations of $\irr(G^\star \mid \chi)$ are of dimension $p^3$.\\

$(iii)$ Similar to $(ii),$ in this case  any representation of $\irr(G^\star \mid  \chi )$ is of dimension $p^2$ if $r_1 = -r_3r_5$ and is of dimension $p^3$ if $r_1\neq -r_3r_5$. \\

%
%
%

$(iv)$  Assume  that $\mu_6\neq 1, \mu_8\neq 1$.
Then by (\ref{stabilizer}), $g = x_1^{a_1}x_2^{a_2}y_1^{b_1}y_2^{b_2}n \in I_{G^\star} (\chi)$ if and only if
$b_1=-r_8a_1, b_2=-r_8a_2+(r_7r_8-r_5)a_1$. Therefore $|I_{G^\star}(\chi)/N| = p^2$ in this case. As earlier, we only need to answer whether $\chi$ extends to $I_{G^\star}(\chi)$ or not.
Now $\chi$ extends to $I_{G^\star}(\chi)$ if and only if
$$\chi\big([x_1^{a_1}x_2^{a_2}y_1^{b_1}y_2^{b_2}z_1^{c_1}z_2^{c_2}, x_1^{a_1'}x_2^{a_2'}y_1^{b_1'}y_2^{b_2'}z_1^{c_1'}z_2^{c_2'}] \big)= 1.$$ By using  \eqref{commutator}
 and substituting the values $\chi(w_i), \chi(v_j)$ for $1 \leq i \leq 4, 1\leq j \leq 2$ in terms of $\chi(u_2)$ we get,
 $$\chi(u_2)^{(a_1'a_2-a_1a_2')\big(r_1-r_2r_8-r_3(r_7r_8-r_5)+r_4r_8^2-\frac{r_8}{2}+\frac{r_8^2}{2}\big)}=1.
$$
Since $a_1, a_1', a_2,a_2'$ are arbitrary, we must have
\begin{eqnarray}
\label{eq:long-degenarte-condition}
r_1-r_2r_8-r_3(r_7r_8-r_5)+r_4r_8^2+{r_8\choose 2}=0.
\end{eqnarray}
As earlier, this implies that $\irr(G^\star \mid \chi)$ consists of $p^2$ dimensional irreducible representations if (\ref{eq:long-degenarte-condition}) holds and of dimension $p^3$ otherwise. This completes our discussion regarding the irreducible representations of $G^\star$.

\subsection{Non-degenerate 2-cocycles of $\hfpp$}\label{nondegco1}In this section we describe the degenerate and non-degenerate cocycles of  $\hfpp$. Recall that upto cohomologous every $2$-cocycle $\alpha$ of $\hfpp$  is given by
\begin{eqnarray}
	\label{2-cocycle-heisenberg-function}
	&&\alpha\big(x_1^{a_1}x_2^{a_2}y_1^{b_1}y_2^{b_2}z_1^{c_1}z_2^{c_2}, x_1^{a_1'}x_2^{a_2'}y_1^{b_1'}y_2^{b_2'}z_1^{c_1'}z_2^{c_2'}\big) =\mu_1^{a_2a_1'} \mu_2^{b_2a_1'}\mu_3^{b_2a_2'} \mu_4^{b_2 b_1'}  \mu_5^{b_1{a_1' \choose 2}+a_1'c_1} \nonumber \\
	&&\hspace{5cm}  \mu_6^{b_2{a_1' \choose 2}+a_1'c_2+a_2'c_1} \mu_7^{ a_1' {b_1\choose 2}-b_1c_1'} \mu_8^{a_2'{b_1\choose 2}-b_2c_1'-b_1c_2'},
\end{eqnarray}
where  $\mu_i \in \bC^\times$ such that $\mu_i^p=1$ for all $i$. Following the notations of Section~\ref{section:projection-function-field}, if $\mu_6 \neq 1$, we assume that $\mu_i=\mu_6^{r_i}$ for $1 \leq i \leq 8$ with $r_6 = 1$
and,  if $\mu_8 \neq 1$ we assume that $\mu_i=\mu_8^{t_i}$ for $1 \leq i \leq 7$ with $t_8 = 1$. Let $G^\star, N, \chi$ be as given in the Section~\ref{section:projection-function-field}.

\begin{thm}\label{nondegenerate}
	The cocycle $\alpha$ as given in \eqref{2-cocycle-heisenberg-function} is  non-degenerate if and only if one of the following holds.
	\begin{enumerate}[label=(\roman*)]		
		\item   $\mu_6=1,$ $ \mu_8 \neq 1$ and $t_4 \neq  t_3t_7$.
		\item     $\mu_8=1$, $\mu_6 \neq 1$ and  $r_1 \neq -r_3r_5$.
		\item    $\mu_8 \neq 1, \mu_6 \neq 1$ and $r_1-r_2r_8-r_3(r_7r_8-r_5)+r_4r_8^2+{r_8\choose 2} \neq 0$.
	\end{enumerate}
In particular, the number of non-degenerate cocycles of  $\hfpp$ is   $p^5(p-1)^2(p+1)$.
\end{thm}
\begin{proof} Let $A =\langle  u_i,v_i, w_j,  i=1,2, 1\leq j \leq 4\rangle \cong (\mathbb Z/p\mathbb Z)^{8}$ be a subgroup of $G^\star$. Then $A$ is a central subgroup of $G^\star$ such that $1 \rightarrow A \rightarrow G^\star \rightarrow G \rightarrow 1$ is a stem extension. For this case, we note that $\tra(\chi|_{A})=[\alpha]$. This fact, Lemma~\ref{lemma:non-degenerate-cocycles} and Section~\ref{section:projection-function-field} give the result.
	\end{proof}

\begin{remark} We remark that the existence of $p^5$ number of non-degenerate cocycles out of  the above $p^5(p-1)^2(p+1)$ has already appeared in \cite[Theorem A]{Ginosar}.


\end{remark}

\subsection{Projective representations of $\mathbb{H}_{3}\big(\mathbb Z/p^2\mathbb Z \big)$} In this section we first give a construction of projective representations of $G = \mathbb H_3(\mathbb Z/p^2\mathbb Z)$. As mentioned earlier, it is enough to give a construction of all ordinary irreducible representations of a representation group of $G$.
From \cite[Theorem 1.2]{PS},
$$
G^\star=\langle x,y\mid [y,x]=z, [z,x]=z_1, [z,y]=z_2, x^{p^2}=y^{p^2}=z^{p^2}=1 \rangle
$$
is a representation group of $G$. Note that,  $x^{a} y^{b} z^{c},  x^{a'} y^{b'} z^{c'}\in G^\star$ satisfy $$
[x^{a} y^{b} z^{c},  x^{a'} y^{b'} z^{c'}]=z^{a'b-ab'}z_1^{b{a'\choose 2}-b'{a\choose 2}+(a'c-ac')}z_2^{a'{b\choose 2}-a{b'\choose 2}+(b'c-bc')}.
$$

Consider $N=\langle z,z_1,z_2\rangle \cong (\mathbb Z/p^2\mathbb Z)^3$, a normal subgroup of $G^\star$.
Any character $\chi:N \to \mathbb C^\times$ is given by  $\chi(z_i)=\lambda_i$ for $i=1,2$ and $\chi(z) = \lambda$, where $\lambda_i^{p^2} = 1$  for $i = 1,2$ and $\lambda^{p^2} = 1$. The inertia group $S_\chi = I_{G^\star}(\chi)$ of $\chi$  consists of the elements $x^ay^bn, n \in N$ such that $\chi([x^ay^b,z])=1$, that is $\chi(z_1)^{a}\chi(z_2)^b=1$.
We consider the various cases of $\lambda_i'$s. \\

(i) Assume $\lambda_1 = \lambda_2 = 1$. In this case $S_\chi = G^\star$. We note that $|S_{\chi}/N| = |G^\star|/|N| = p^4$, $S_\chi/N$ is abelian. Hence by Clifford theory, depending on $\chi(z)$, $\irr(G^\star \mid \chi)$ consists of representations of either $1$, $p$ or $p^2$ dimension. \\

(ii) Assume $\lambda_1=1,\lambda_2\neq 1$ such that $\lambda_2^p = 1$. Then  $S_\chi=\langle N, x, y^{p} \rangle$. Therefore $\vert S_\chi/N\vert = p^3$ and $S_\chi/N$ is abelian. Hence $\irr(G^\star \mid \chi)$ consists of $p$-dimensional representations if $\chi$ extends to $S_\chi$ and if $\chi$ does not extend to $S_\chi$ then  $\irr(G^\star \mid \chi)$ consists to $p^2$ dimensional representations. We use Lemma~\ref{lemma:extension-condition} and  $\chi(z_2)^{pt}=\chi(z_2)^{pt_1}=1$ to observe that $\chi$ extends to $S_\chi$ if and only if
$$
\chi([x^{a} y^{pt} z^{c},  x^{a'} y^{pt_1} z^{c'}])=\chi(z)^{p(ta'-t_1a)}=1.
$$
Therefore $\chi$ extends if and only if $\lambda^p = 1$. \\

(iii) Assume $\lambda_1=1$ and order of $\lambda_2$ is $p^2$. In this case, we obtain $S_\chi = \langle N, x \rangle$. Since $S_\chi/N$ is cyclic of order $p^2$.  By Clifford theory, $\irr(G^\star\mid \chi) $ consists of $p^2$ dimensional irreducible representations. \\

(iv) For $\lambda_2 = 1$ and $\lambda_1 \neq 1$, we proceed as above. \\

(iv) Suppose $\lambda _1 \neq 1,\lambda_2 \neq 1$ and $\lambda_1^p = \lambda_2^p = 1$.  Assume that $\chi(z_1)=\chi(z_2)^r$ for some $1\leq r< p-1$. For the inertia group computations,  $a$ and $b$ are such that $\chi(z_1)^a\chi(z_2)^b=1$, that is $\chi(z_2)^{b+ra}=1$. This implies
$b \in \{-ra+pt \mid 0\leq t \leq p-1 \}$. Therefore $S_\chi=\langle N, x^ay^{-ra+pt}\rangle $. So $\vert S_\chi/N \vert = p^3$ and $S_\chi/N$ is abelian. As in (ii),  $\irr(G^\star \mid \chi)$ consists of $p$-dimensional representations if $\chi$ extends to $S_\chi$ and if $\chi$ does not extend to $S_\chi$ then  $\irr(G^\star \mid \chi)$ consists to $p^2$ dimensional representations. So it remains to determine when the conditions for which $\chi$ extend to $S_\chi$.
By Lemma~\ref{lemma:extension-condition}, $\chi$ extends to $S_\chi$ if and only if $\chi([x^{a} y^{-ra+pt} z^{c},  x^{a'} y^{-ra'+pt'} z^{c'}])=1$. This is equivalent to
\[
\chi(z)^{p(a't-at')} \chi(z_2)^{r^2aa'(a-a')} = 1,
\]
where $a, a', t, t'$ are arbitrary. Therefore, we must have $\lambda_2^{r^2} = 1$ and $\chi(z)^p = 1$.  This implies that $r$ is a multiple of $p$. This in turn gives that $\lambda_1 = \lambda_2^r = 1$, a contradiction. Hence $\chi$ does not extend to $S_\chi$ in this case. Therefore $\irr(G^\star \mid \chi)$ consists to $p^2$ dimensional representations. \\

(v) Assume that $\lambda_1 \neq 1$ and $\lambda_2$ is of order $p^2$. Assume that $\lambda_1 = (\lambda_2)^r$. It is easy to see that $S_\chi = \langle N, x^{a} y^{-ra} \rangle$, and therefore  $S_\chi/N$ is an abelian group of order $p^2$. As in (iv), we obtain that $\chi$ extends to  $S_\chi$ if and only if $\lambda_1^p = 1$. Hence in this case $\irr(G^\star \mid \chi)$ consists of $p^2$-dimensional representations. If both $\lambda_1$ and $\lambda_2$ are of order $p^2$, then  $\irr(G^\star \mid \chi)$ consists of $p^3$-dimensional representations. \\

(vi) Assume that order of $\lambda_1$ is $p^2$ and that of $\lambda_2$ is $p$. This case can be done parallel to (iv). Here we obtain that $\irr(G^\star \mid \chi)$ consists of $p^2$ dimensional representations.

\subsection{Non-degenerate 2-cocycles of $\mathbb{H}_{3}\big(\mathbb Z/p^2\mathbb Z \big)$}\label{nondegco2}
In this section we describe the non-degenerate cocycles of the group $\mathbb H_3(\mathbb Z/p^2\mathbb Z)$.
\begin{thm}
The cocycle $\alpha$ as given in \eqref{cocycledescription} is  non-degenerate if and only if $\lambda_1$ and $\lambda_2$ are of order $p^2$. In particular, the number of non-degenerate cocycles is $(p-1)^2 p^2$.
\end{thm}

\begin{proof}

Recall that every $2$-cocycle  of   $G = \mathbb H_3(\mathbb Z/p^2\mathbb Z)$, upto cohomologous, is given by
\begin{eqnarray}\label{cocycledescription}
\alpha\big(x^{a_1}y^{b_1}z^{c_1}, x^{a_1'}y^{b_1'}z^{c_1'}\big)=\lambda_1^{c_1'a_1+b_1{a_1'\choose 2}+a_1b_1a_1'}\lambda_2^{c_1'b_1+a_1'{b_1\choose 2}},  \lambda_1^{p^2}= \lambda_2^{p^2}=1.
\end{eqnarray}
Let $Z $ be a subgroup of $G^\star$ generated by $z_1$ and $z_2$. Then $Z \subseteq N$ and $1 \rightarrow Z \rightarrow G^\star \rightarrow G \rightarrow 1$ is a stem extension. For any character $\chi$ of $N$ as given above, let $\chi|_Z$ denote its restriction to the group $Z$. We note that $\tra(\chi|_{Z})=[\alpha]$. The result now follows from above description of ordinary irreducible representations of $G^\star$, Theorem~\ref{inf} and Lemma~\ref{lemma:non-degenerate-cocycles}.
\end{proof}
\noindent {\bf Acknowldgement:}
 SH acknowledges the financial support received from IISc in the form of Raman postdoctoral fellowship (R(IA)/CVR-PDF/2020/2700). EKN thanks SERB, India for the financial support through MATRICS grant MTR/2018/000501. PS thanks SERB, India and MHRD, India for the financial support through MATRICS grant \\
 MTR/2018/000094 and MHRD-SPARC grant (SPARC/2018-
 2019/P88/SL).
\bibliographystyle{amsplain}
\bibliography{Heisenberg-over-rings}

\end{document}